\documentclass[10pt]{amsart}

\usepackage{graphicx}
\usepackage{amsfonts,amsmath,latexsym,amssymb,amsthm}
\usepackage{verbatim}
\usepackage{mathrsfs}

\newtheorem{theorem}{Theorem}[section]

\newtheorem{proposition}[theorem]{Proposition}
\newtheorem{corollary}[theorem]{Corollary}

\theoremstyle{definition}

\theoremstyle{remark}

\numberwithin{equation}{section}

\newcommand{\Oseen}{\mathcal{O}}

\newcommand{\IP}{\mathbb{P}}

\newcommand{\IQ}{\mathbb{Q}}
\newcommand{\Qbar}{\overline{\IQ}}

\newcommand{\B}{\mathscr{B}}
\newcommand{\D}{D_K}
\newcommand{\del}{\eta(K)}

\newcommand{\ka}{\kappa}

\begin{document}

\title{Bounds for the $\ell$-torsion in class groups}

\author{Martin Widmer}
\address{Department of Mathematics\\ 
Royal Holloway, University of London\\ 
TW20 0EX Egham\\ 
UK}
\email{martin.widmer@rhul.ac.uk}

\date{\today, and in revised form ....}

\dedicatory{}


\begin{abstract}
We prove for each integer $\ell\geq 1$ an unconditional upper bound for the
size of the $\ell$-torsion subgroup $Cl_K[\ell]$ of the class group of $K$, which holds for all but a zero density set of number fields $K$
of degree $d\in\{4,5\}$ (with the additional restriction in the case $d = 4$ that the
field be non-$D_4$). For sufficiently large $\ell$ this improves recent results of Ellenberg, Matchett Wood, and Pierce,
and is also stronger than the best currently known pointwise bounds under GRH. 
Conditional on GRH and on a weak conjecture on the distribution of number fields our bounds
also hold for arbitrary degrees $d$. 
\end{abstract}

\maketitle



\section{Introduction}\label{introductionchapter3}

In this article we prove for each integer $\ell\geq 1$ an unconditional upper bound for the
size of the $\ell$-torsion subgroup $Cl_K[\ell]$ of the class group of $K$, which holds for all but a zero density set of number fields $K$
of degree $d\in\{4,5\}$ (with the additional restriction in the case $d = 4$ that the
field be non-$D_4$). For sufficiently large $\ell$ these results improve results of Ellenberg, Matchett Wood, and Pierce \cite{EllenbergPierceWood},
and are also stronger than the best currently known pointwise bounds assuming GRH due to Ellenberg and Venkatesh \cite{EllVentorclass}. 
Conditional on GRH and on a weak conjecture on the distribution of number fields our bounds
also hold for arbitrary degrees $d$. 

We always assume $X\geq 2$, and that $\ell$ is a positive integer. We shall use the $O(\cdot)$, $\ll$, and $\gg$ notation; throughout
the implied constants will depend only on the indicated parameters.
Denote the modulus of the discriminant of the number field $K$ by $\D$, and its degree $[K:\IQ]$ by $d$.

Bounding $\#Cl_K[\ell]$ by the size of the full class group, and using a classical bound (see, e.g., \cite[Thm 4.4]{Nark1990})  yields the trivial bound\footnote{As usual $\varepsilon$ denotes an arbitrarily small positive number.} 
\begin{alignat}1\label{trivial}
\#Cl_K[\ell] \ll_{d,\varepsilon} \D^{1/2+\varepsilon}.
\end{alignat}
While it is conjectured (see, e.g., \cite[Conjecture 1.1]{EllVentorclass}, \cite[Section 3]{Duke} and \cite{Zhang05})  that
$$\#Cl_K[\ell] \ll_{d,\ell,\varepsilon} \D^{\varepsilon},$$
unconditional nontrivial bounds that hold for all number fields of degree $d$
are known only for $\ell=2$, and for $d\leq 4$ and $\ell=3$.
For 
$d=\ell=2$ the conjecture follows from Gauss' genus theory,
whereas for $(d, \ell)=(2,3)$ the first nontrivial bounds were obtained by Pierce \cite{Pierce05, Pierce06}, and Helfgott and Venkatesh  \cite{HelfgottVenkatesh}. Currently 
the best bound is
$$\#Cl_K[3] \ll_{\varepsilon} \D^{1/3+\varepsilon}$$
due to Ellenberg and Venkatesh  \cite{EllVentorclass} which holds
also for cubic fields. Moreover, they established a nontrivial bound for quartic fields
(with, e.g., an exponent $83/168+\varepsilon$ provided $K$ is an $S_4$ or $A_4-$field).
Another, very  recent, breakthrough due to Bhargava, Shankar, Taniguchi, Thorne, Tsimerman, and Zhao \cite{Bhargava2tor} provides for arbitrary $d$ the bound
$$\#Cl_K[2] \ll_{d,\varepsilon} \D^{1/2-1/(2d)+\varepsilon},$$
and for $d\in \{3,4\}$ they can even take 
the exponent $0.2784$. 

Regarding general $\ell$ there are only conditional results, assuming GRH.  
The latter is used to guarantee the existence of many small splitting primes;  the
idea of using these to investigate torsion in class groups has been around for a while,
see, e.g., \cite{BoydKisilevsky, Soundararajan}.  However, Ellenberg and Venkatesh \cite{EllVentorclass} have greatly extended 
this strategy, and proved the first (although conditional on GRH) nontrivial general upper bound   
\begin{alignat}1\label{generalbound}
\#Cl_K[\ell] \ll_{d,\ell,\varepsilon} \D^{1/2-1/(2\ell(d-1))+\varepsilon}.
\end{alignat}

So much for pointwise bounds; regarding results on the average Davenport and Heilbronn \cite{81} showed that
$$\sum_{[K:\IQ]=2 \atop \D\leq X}\#Cl_K[3] \sim 5/(3\zeta(2))X,$$
and Bhargava \cite{Bhargava05} established the asymptotics for $2$-torsion in cubic fields
$$\sum_{[K:\IQ]=3 \atop \D\leq X}\#Cl_K[2] \sim 23/(16\zeta(3))X.$$
No other asymptotics  are known but Heath-Brown and Pierce \cite{Heath-BrownPierce} proved that for primes $\ell\geq 5$  
$$\sideset{}{'}\sum_{[K:\IQ]\leq 2 \atop \D\leq X}\#Cl_K[\ell] \ll_{\ell,\varepsilon} X^{3/2-3/(2\ell+2)+\varepsilon},$$
where the apostrophe indicates that the sum is restricted to imaginary quadratic fields.
Recently Ellenberg, Matchett Wood, and Pierce \cite[Corollary 1.1.1 and 1.1.2]{EllenbergPierceWood} established the first 
nontrivial  unconditional average bounds for arbitrary $\ell$.
\begin{corollary}[Ellenberg, Matchett Wood, and Pierce]\label{CorEPW}
Suppose $d\in \{2,3,4,5\}$, $\ell(2)=\ell(3)=1$, $\ell(4)=8$, $\ell(5)=25$,
and $\varepsilon>0$.  As $K$ ranges over degree $d$ number fields with $\D\leq X$ (and non-$D_4$ in the case $d = 4$), we have
$$\sum_K\#Cl_K[\ell] \ll_{\ell,\varepsilon} X^{3/2-1/(2\ell(d-1))+\varepsilon}$$
if $\ell\geq \ell(d)$, and we have
$$\sum_K\#Cl_K[\ell] \ll_{\ell,\varepsilon} X^{3/2-\delta_0(d)+\varepsilon}$$
if $\ell<\ell(d)$, where $\delta_0(4)=1/48$ and $\delta_0(5)=1/200$.
\end{corollary}

Corollary \ref{CorEPW} is an immediate consequence of their main result \cite[Theorem 1.1]{EllenbergPierceWood} (see Theorem 1.2 below) which is {\em unconditional},
and gives upper bounds of the same size as the conditional result (\ref{generalbound}) outside a family of density zero, provided $d\leq 5$ and $\ell$ is sufficiently large.
To state their result in a uniform way we additionally  set $\delta_0(2)=\delta_0(3)=1$.

\begin{theorem}[Ellenberg, Matchett Wood, and Pierce]\label{ThmEPW}
Suppose $d\in \{2,3,4,5\}$,  and $\varepsilon>0$. Then for all but 
$O_{\ell,\varepsilon}(X^{1-\min\{1/(2\ell(d-1)),\delta_0(d)\}+\varepsilon})$ 
degree $d$ number fields $K$
with $\D\leq X$  (and non$-D_4$ when $d=4$) we have
$$\#Cl_K[\ell] \ll_{\ell,\varepsilon} \D^{1/2-\min\{1/(2\ell(d-1)),\delta_0(d)\}+\varepsilon}.$$
\end{theorem}
Very roughly speaking, Ellenberg, Matchett Wood, and Pierce's strategy is to show that ``most fields'' have sufficiently many small splitting primes,
and  then to apply the general strategy of Ellenberg and Venkatesh (see Proposition \ref{keylemma} in Section \ref{sectionkeylemma}).
We follow this approach but combine it with a new idea, showing that for ``most'' number fields the smallest height of a primitive
element is significantly bigger than in the worst case scenario, at least\footnote{Actually, we do not know whether for ``most'' cubic fields the smallest generator is significantly bigger than in the worst case scenario,
but Ruppert \cite[Proposition 2]{Rupp} has shown that this is definitely not the case for quadratic fields. Consequently, we only get an improvement if $d\geq 4$ although
Theorem \ref{Thm1} and its proof remain valid for $d\in \{2,3\}$.} when $d\geq 4$.
Therefore, our main result  improves (see Corollary \ref{Cor2}) the quartic and quintic case of Theorem \ref{ThmEPW} 
and consequently Corollary \ref{CorEPW} (see Corollary \ref{Cor3}).
\begin{theorem}\label{Thm1}
Suppose $d\in \{4,5\}$,  $0<\gamma\leq 1/(d+1)$, and $\varepsilon>0$. Then for all but 
$O_{\ell,\gamma,\varepsilon}(X^{1-\min\{\gamma/\ell,\delta_0(d)\}+\varepsilon}+X^{\gamma(d+1)})$ 
degree $d$ number fields $K$
with $\D\leq X$  (and non$-D_4$ when $d=4$) we have
$$\#Cl_K[\ell] \ll_{\ell,\gamma,\varepsilon} \D^{1/2-\min\{\gamma/\ell,\delta_0(d)\}+\varepsilon}.$$
\end{theorem}
We set $\ell_4=10$ and $\ell_5=34$.
With  $\gamma=\ell/(\ell(d+1)+1)$ if $\ell\geq \ell_d$ and $\gamma=(1-\delta_0(d))(d+1)$ otherwise we have
$1-\min\{\gamma/\ell,\delta_0(d)\}=\gamma(d+1)$, and hence Theorem \ref{Thm1} yields immediately the following corollary
which improves the quartic and quintic case of the main result \cite[Theorem 1.1]{EllenbergPierceWood} when $\ell>7$ ($d=4$) and $\ell>24$ ($d=5$)
respectively.
\begin{corollary}\label{Cor2}
Suppose $d\in \{4,5\}$,  and $\varepsilon>0$. 
If $\ell\geq \ell_d$ then for all but 
$O_{\ell,\varepsilon}(X^{1-1/(\ell(d+1)+1)+\varepsilon})$
degree $d$ number fields $K$
with $\D\leq X$  (and non$-D_4$ when $d=4$) we have
$$\#Cl_K[\ell] \ll_{\ell,\varepsilon} \D^{1/2-1/(\ell(d+1)+1)+\varepsilon}.$$
If $\ell<\ell_d$ then for all but 
$O_{\ell,\varepsilon}(X^{1-\delta_0(d)+\varepsilon})$
degree $d$ number fields $K$
with $\D\leq X$  (and non$-D_4$ when $d=4$) we have
$$\#Cl_K[\ell] \ll_{\ell,\varepsilon} \D^{1/2-\delta_0(d)+\varepsilon}.$$
\end{corollary}
Corollary \ref{Cor2} and dyadic summation, using the trivial bound (\ref{trivial}) for the exceptional fields, yields our next result which improves the quartic and quintic case
of \cite[Corollary 1.1.1]{EllenbergPierceWood}.
\begin{corollary}\label{Cor3}
Suppose $d\in \{4,5\}$, and $\varepsilon>0$. 
As $K$ ranges over degree $d$ number fields with $\D\leq X$ (and non-$D_4$ in the case $d = 4$), we have
$$\sum_K\#Cl_K[\ell] \ll_{\ell,\varepsilon} X^{3/2-1/(\ell(d+1)+1)+\varepsilon}$$
if $\ell\geq \ell_d$, and we have
$$\sum_K\#Cl_K[\ell] \ll_{\ell,\varepsilon} X^{3/2-\delta_0(d)+\varepsilon}$$
if $\ell<\ell_d$.
\end{corollary}

Bhargava \cite{Bhargava05, Bhargava10} proved that the conjectured asymptotics
$c_dX$  for the number of degree $d$ fields $K$ with $\D\leq X$ holds true for $d\in \{4,5\}$,
and in the case $d=4$ also (with a different positive constant) when restricting to non-$D_4$ fields.
Thus, applying Theorem \ref{Thm1} with $\gamma=1/(d+1)-\varepsilon$ we get the following corollary.

\begin{corollary}\label{Cor1}
Suppose $d\in \{4,5\}$, and $\varepsilon>0$. If $\ell\geq \ell_d$ then for $100\%$ of non-$D_4$ degree $d$ fields (when enumerated by modulus of the discriminant) we have 
$$\#Cl_K[\ell] \ll_{\ell,\varepsilon} \D^{1/2-1/(\ell(d+1))+\varepsilon}.$$
Moreover, if $1\leq \ell<\ell_d$ then  for $100\%$ of non-$D_4$ degree $d$ fields we have 
$$\#Cl_K[\ell] \ll_{\varepsilon} \D^{1/2-\delta_0(d)+\varepsilon}.$$
\end{corollary}

We now turn to more general but conditional results.
\begin{theorem}\label{Thm2}
Suppose $m\mid d$, $F$ is a number field of degree $m$, and $\varepsilon>0$. 
Assume GRH, and additionally that there are $\gg_{F,d} X$ number fields $K$ of degree $d$ with $\D\leq X$, and containing $F$.
Then for $100\%$ of  degree $d$ fields $K$ containing $F$ we have 
$$\#Cl_K[\ell] \ll_{d,\ell,\varepsilon} \D^{1/2-m/(\ell(m+d))+\varepsilon}.$$
\end{theorem}
So here we get an exponent which depends only on $d/m$ (the degree of $K/F$) but is independent of the degree of $K/\IQ$.
The proof of Theorem \ref{Thm2} actually provides a more precise quantitative result analogous to Theorem \ref{Thm1}.

Datskovsky and Wright
\cite[Theorem 4.2 and Theorem 1.1]{82} have shown that there are
$\gg_{F,d} X$ number fields $K$ of degree $d$ with $\D\leq X$, and containing $F$,
provided $2\leq d/m\leq 3$, and recent work of Bhargava, Shankar, and Wang \cite[Theorem 1.1]{BhargavaShankarWang} implies that
this holds even for $2\leq d/m\leq 5$. Hence, we get the following corollary.
\begin{corollary}\label{cor2}
Suppose $m\mid d$, $2\leq d/m\leq 5$, $F$ is a number field of degree $m$, and $\varepsilon>0$. 
Assume GRH. Then for $100\%$ of  degree $d$ fields $K$ containing $F$ we have 
$$\#Cl_K[\ell] \ll_{d,\ell,\varepsilon} \D^{1/2-m/(\ell(m+d))+\varepsilon}.$$
\end{corollary}

\section{Ellenberg and Venkatesh's Key Lemma}\label{sectionkeylemma}
Let 
\begin{alignat*}1
H_K(\alpha)=\prod_{v\in M_K}\max\{1,|\alpha|_v\}^{d_v}
\end{alignat*}
be the relative multiplicative Weil height of $\alpha\in K$.
Here $M_K$ denotes the set of places of $K$, and for each place $v$ we choose the unique representative $| \cdot |_v$ 
that either extends the usual Archimedean absolute value on $\IQ$ or a usual $p$-adic absolute value on $\IQ$, and $d_v = [K_v : \IQ_v]$ denotes the local degree at $v$.
Note that this is exactly the height  in \cite[(2.2)]{EllVentorclass} for the principal divisor $(\alpha, (\alpha))$ associated to $\alpha\in K^\times$.
We also use the following invariant
\begin{alignat*}1
\del=\inf\{H_K(\alpha);K=\IQ(\alpha)\},
\end{alignat*}
introduced by Roy and Thunder \cite{8}, and also studied\footnote{In the cited works the author used the absolute instead of the relative height, and denoted the invariant by $\delta(K)$.} in \cite{VaalerWidmer2013,VaalerWidmer}.
First we use the fact that the proof of the key lemma \cite[Lemma 2.3]{EllVentorclass} of Ellenberg and Venkatesh proves actually the following stronger
statement. Recall from \cite{EllVentorclass} that a prime ideal $\mathfrak{p}$ of $\Oseen_K$ is said to be an extension of a prime ideal 
from a subfield $K_0\subsetneq K$ if there exists a prime ideal $\mathfrak{p}_0$ of $\Oseen_{K_0}$ such that $\mathfrak{p}=\mathfrak{p_0}\Oseen_K$.
If $\mathfrak{p}$ and $\mathfrak{p}_0$ are non-zero prime ideals in $\Oseen_K$ and $\Oseen_{K_0}$ respectively and $\mathfrak{p}\mid \mathfrak{p}_0\Oseen_K$ 
then we say $\mathfrak{p}$ is unramified in $K/K_0$
if $\mathfrak{p}^2\nmid \mathfrak{p}_0\Oseen_K$. 

\begin{proposition}[Ellenberg and Venkatesh]\label{keylemma}
Suppose $K$ is a number field of degree $d$, $\del>\D^\gamma$, $\delta<\gamma/\ell$, and $\varepsilon>0$.
Moreover, suppose $\mathfrak{p}_1,\ldots,\mathfrak{p}_M$ are $M$ prime ideals in $\Oseen_K$ of norm
$N(\mathfrak{p}_i)\leq \D^\delta$ that are unramified in $K/\IQ$ and are not extensions of prime ideals from any proper subfield of $K$. 
Then we have 
$$\#Cl_K[\ell] \ll_{d,\ell,\gamma,\varepsilon} \D^{1/2+\varepsilon}M^{-1}.$$
\end{proposition}
\begin{proof}
Exactly as in \cite[Lemma 2.3]{EllVentorclass} with $K_0=\IQ$ except that we replace their Lemma 2.2 by the hypothesis $\del>\D^\gamma$.
\end{proof}
Ellenberg \cite[Proposition 1]{Elllowheight} pointed out that the proof of \cite[Lemma 2.3]{EllVentorclass} even 
provides the stronger conclusion $\#Cl_K[\ell] \ll_{d,\ell,\varepsilon} \D^{1/2+\varepsilon}M_K$, where
$M_K:=\inf_T(T^{-1/\ell}(1+N'_K(T)))$, and $N'_K(T)$ denotes the number of primitive elements in $K$ of 
(relative) height at most $T$.

\section{The main proposition}

Let $d>1$ be an integer. We set
\begin{alignat}1\label{coll}
S_{\IQ,d}=\{K\subset \Qbar; [K:\IQ]=d\}
\end{alignat}
for the collection of all number fields of degree $d$. For a subset $S\subset S_{\IQ,d}$
we set 
\begin{alignat*}1
\B_S(X;Y,M)&:=\{K\in S;\D\leq X, \text{ at most $M$ primes $p\leq Y$ split completely in $K$}\},\\
P_S&:=\{\alpha\in \Qbar;\IQ(\alpha)\in S\},\\
N_H(P_S,X)&:=\#\{\alpha\in P_S;H_{\IQ(\alpha)}(\alpha)\leq X\}.
\end{alignat*}
We can now formulate our main proposition.

\begin{proposition}\label{mainprop}
Suppose $S\subset S_{\IQ,d}$ and $\theta>0$ are such that
\begin{alignat*}1
N_H(P_S,X)\ll_{S,\theta} X^\theta.
\end{alignat*}
Let $\gamma>0$, $\varepsilon>0$, $\tilde{\delta}_0>0$, $\delta_0:=\min\{\gamma/\ell-2\varepsilon,\tilde{\delta}_0\}$, and $E_{\delta_0,\varepsilon}(\cdot)$ be an increasing function such that 
\begin{alignat*}1
\B_S(X;X^{\delta_0},X^{\delta_0-\varepsilon})\leq E_{\delta_0,\varepsilon}(X).
\end{alignat*}
Then we have 
$$\#Cl_K[\ell] \ll_{d,\ell,\gamma,\varepsilon} \D^{1/2-\delta_0+2\varepsilon}$$
for all but $O_{S,\theta}((\log X)E_{\delta_0,\varepsilon}(X)+X^{\gamma\theta})$ fields $K$ in $S$ with $\D\leq X$.
\end{proposition}
\begin{proof}
We set $\ka_i=\log_2X-[\log_2X]+i$, and for $1\leq i\leq [\log_2X]$
$$D_{\delta_0,\gamma}(i)=\{K\in S;2^{\ka_{i-1}}<\D\leq 2^{\ka_i}, \del>\D^\gamma, K\notin \B_S(2^{\ka_i};2^{\ka_i\delta_0},2^{\ka_i(\delta_0-\varepsilon)})\}.$$
By Hermite's Theorem the inequality (\ref{mainpropbound}) below trivially holds true for all $K\in S$ with $\D\leq 2^{\delta_0/\varepsilon}$; so let's assume $\D>2^{\delta_0/\varepsilon}$.
Note that for $K\in D_{\delta_0,\gamma}(i)$  there exist $\geq \D^{\delta_0-\varepsilon}$ primes $p\leq (2\D)^{\delta_0}<\D^{\delta_0+\varepsilon}$
that split completely in $K$. Since $\delta_0+\varepsilon<\gamma/\ell$ we can apply Proposition \ref{keylemma} with $\delta=\delta_0+\varepsilon$ and $M=\lceil\D^{\delta_0-\varepsilon}\rceil$. 
Hence, we have shown that
\begin{alignat}1\label{mainpropbound}
\#Cl_K[\ell] \ll_{d,\ell,\gamma,\varepsilon} \D^{1/2-\delta_0+2\varepsilon}
\end{alignat}
for all $K\in \cup_i D_{\delta_0,\gamma}(i)$. 
Next, we note that
\begin{alignat*}1
&\#\cup_i D_{\delta_0,\gamma}(i)\\
&\geq \#\{K\in S; \D\leq X\}-\sum_i\#\B_S(2^{\ka_i};2^{\ka_i\delta_0},2^{\ka_i(\delta_0-\varepsilon)})-\#\{K\in S;\D\leq X, \del\leq \D^\gamma\}.
\end{alignat*}
By hypothesis $\#\B_S(2^{\ka_i};2^{\ka_i\delta_0},2^{\ka_i(\delta_0-\varepsilon)})\leq E_{\delta_0,\varepsilon}(2^{\ka_i})$, and since $2^{\ka_i}\leq X$ we conclude
\begin{alignat*}1
\sum_i\#\B_S(2^{\ka_i};2^{\ka_i\delta_0},2^{\ka_i(\delta_0-\varepsilon)})\leq (\log_2 X)E_{\delta_0,\varepsilon}(X).
\end{alignat*}
Finally, we observe that the image of the map $\alpha\rightarrow \IQ(\alpha)$ with domain 
$$\{\alpha\in P_S;H_{\IQ(\alpha)}(\alpha)\leq X^\gamma\}$$ 
covers the set
$$\{K\in S;\D\leq X, \del\leq \D^\gamma\},$$ 
and using the hypothesis we conclude that
\begin{alignat*}1
\#\{K\in S;\D\leq X, \del\leq \D^\gamma\}\leq N_H(P_S,X^\gamma)\ll_{S,\theta}X^{\gamma\theta}.
\end{alignat*}
Hence, we have shown that for all but $O_{S,\theta}((\log X)E_{\delta_0,\varepsilon}(X)+X^{\gamma\theta})$ fields $K$ in $S$ with $\D\leq X$ we have
$\#Cl_K[\ell] \ll_{d,\ell,\gamma,\varepsilon} \D^{1/2-\delta_0+2\varepsilon}$.
\end{proof}

\section{Proofs of the Theorems}\label{proofs}
\subsection{Upper bounds for $N_H(P_S,X)$}\label{upperbound}
Let $F$ be a number field of degree $m\mid d$, and define
$$S_{F,d}:=\{K\subseteq \Qbar; [K:\IQ]=d, F\subseteq K\}.$$ 
Applying Schmidt's \cite[Theorem]{22} with Schmidt's $(K,k,d,n)$ replaced by our $(F,m,d/m,1)$
shows that the number of $P=(1:\alpha)\in \IP^1(\Qbar)$ with $[F(\alpha):F]=d/m$ and $H_{F}(P)\leq X$ (for Schmidt's
projective field height \cite[(1.2)]{22}) is $\ll_{m,d}X^{d/m+1}$. Since $H_{F}(P)=H_{F(\alpha)}(\alpha)$
we conclude that
\begin{alignat*}1
\#\{\alpha\in \Qbar; [F(\alpha):F]=d/m, H_{F(\alpha)}(\alpha)\leq X\}\ll_{m,d}X^{d/m+1}. 
\end{alignat*}
Note that if $\alpha \in P_{S_{F,d}}$ then $[F(\alpha):F]=d/m$ and $\IQ(\alpha)=F(\alpha)$ so that $H_{\IQ(\alpha)}(\alpha)=H_{F(\alpha)}(\alpha)$. Therefore,
\begin{alignat}1\label{ptsupperbound}
N_H(P_{S_{F,d}},X)\ll_{m,d}X^{d/m+1}.
\end{alignat}

\subsection{Proof of Theorem \ref{Thm1}}
Let $S^*_{\IQ,4}$, the set of all non-$D_4$ number fields of degree $4$.
We apply Proposition \ref{mainprop} with $S=S^*_{\IQ,4}$ if $d=4$, and with $S=S_{\IQ,5}$ if $d=5$.
As explained in Section \ref{upperbound} we can take 
$\theta=d+1$. Let $\varepsilon>0$, and set $\tilde{\delta}_0:=\delta_0(d)$ so that $\delta_0=\min\{\gamma/\ell-2\varepsilon,\delta_0(d)\}$.
By \cite[Proposition 6.1]{EllenbergPierceWood} and \cite[Theorems 2.2 and Theorem 2.3]{EllenbergPierceWood} we have\footnote{Here , as in \cite[Proposition 7.1]{EllenbergPierceWood},  $c_0(d)$ denotes a positive constant depending only on $d$.} (cf. \cite[Proposition 7.1]{EllenbergPierceWood})
$$\#\B_S(X;X^{\delta_0},\frac{c_0(d)}{2}X^{\delta_0}(\log X^{\delta_0})^{-1})\ll_{\varepsilon} X^{1-\delta_0+\varepsilon},$$
provided $X$ is sufficiently large in terms of $\delta_0$.
Hence,
$$\#\B_S(X;X^{\delta_0},X^{\delta_0-\varepsilon})\ll_{\ell,\gamma,\varepsilon} X^{1-\delta_0+\varepsilon}.$$
Thus we can take $E_{\delta_0,\varepsilon}(X)=C_{\ell,\gamma,\varepsilon}X^{1-\delta_0+\varepsilon}$ for a sufficiently large constant $C_{\ell,\gamma,\varepsilon}$.
We conclude from Proposition \ref{mainprop} that for all but $O_{\ell,\gamma,\varepsilon}(X^{1-\delta_0+2\varepsilon}+X^{\gamma(d+1)})$ fields $K\in S$
with $\D\leq X$ we have 
$\#Cl_K[\ell] \ll_{\ell,\gamma,\varepsilon} \D^{1/2-\delta_0+2\varepsilon}$.
Since $\delta_0\geq\min\{\gamma/\ell,\delta_0(d)\}-2\varepsilon$ and $\varepsilon>0$ was arbitrary the statement of Theorem \ref{Thm1} follows.

\subsection{Proof of Theorem \ref{Thm2}}
We will apply Proposition \ref{mainprop} with $S=S_{F,d}$.
From Section \ref{upperbound} we known that we can take  $\theta=d/m+1$. 
Let $0<\varepsilon<m/(3\ell(d+m))$, choose $\tilde{\delta}_0=1$ and $\gamma=m/(d+m)-\ell\varepsilon$,  
so that $\delta_0=\min\{\gamma/\ell-2\varepsilon,\tilde{\delta}_0\}=m/(\ell(d+m))-3\varepsilon>0$. Since we assume GRH we can apply Lagarias and Odlyzko's effective Chebotarev density Theorem \cite{lo1977} to the normal closure of $K\in S$ to deduce\footnote{We use that the degree of the normal closure $L$ of $K$
is at most $d!$ and that $\log D_L\leq 2(d!)^2\log \D$ since each prime that ramifies in $L$ must ramify in $K$, and the order to which a rational prime divides $D_L$ is bounded from above by $2[L : Q]^2$ (cf. \cite[Theorem B.2.12.]{BG}).} that 
for every $K\in S$  the number of primes $p\leq Y$ that split completely in $K$ is $>Y^{1-\varepsilon'}$, provided $Y\geq (\log\D)^2$ and $Y\geq Y_0(F,d,\varepsilon')$.
Setting  $Y=X^{\delta_0-\varepsilon}$ and $\varepsilon'=\varepsilon/\delta_0$
we conclude that for all $X$ large enough the set $\B_S(X;Y,Y^{1-\varepsilon'})$ is empty, and hence we have for all $X\geq 2$
$$\#\B_S(X;X^{\delta_0},X^{\delta_0-\varepsilon})\ll_{F,d,\ell,\varepsilon}1.$$
Now we can  apply Proposition \ref{mainprop} to conclude that for all but $O_{F,d,\ell,\varepsilon}(X^{1-\ell\varepsilon(d+m)/m})$
fields $K$ of degree $d$ with $\D\leq X$ and containing $F$ we have 
$\#Cl_K[\ell] \ll_{d,\ell,\varepsilon} \D^{1/2-m/(\ell(d+m))+5\varepsilon}.$
By hypothesis there are $\gg_{F,d} X$ number fields $K$ of degree $d$ with $\D\leq X$, and containing $F$. 
Since $\varepsilon>0$ was arbitrarily small the statement of Theorem \ref{Thm2} follows.


\bibliographystyle{amsplain}
\bibliography{literature}

\end{document}